\documentclass[11pt]{amsart}
\usepackage{amsmath,amssymb,latexsym}
\usepackage[small]{caption}
\usepackage{graphicx,color,mathrsfs,tikz}
\usepackage{subfigure,color}
\usepackage{cite}
\usepackage[colorlinks=true,urlcolor=blue,
citecolor=red,linkcolor=blue,linktocpage,pdfpagelabels,
bookmarksnumbered,bookmarksopen]{hyperref}
\usepackage[italian,english]{babel}
\usepackage{units}
\usepackage{enumitem}
\usepackage[left=2.65cm,right=2.65cm,top=2.8cm,bottom=2.8cm]{geometry}
\usepackage[hyperpageref]{backref}

\usepackage[colorinlistoftodos,prependcaption]{todonotes}

\numberwithin{equation}{section}
\newtheorem{theorem}{Theorem}[section]

\newtheorem{remark}[theorem]{Remark}

\theoremstyle{definition}

\renewcommand{\epsilon}{\eps}
\renewcommand{\i}{{\rm i}}

\newcommand{\C}{{\mathbb C}}

\newcommand{\R}{{\mathbb R}}

\newcommand{\eps}{\varepsilon}

\newcommand{\pnorm}[2][]{\if #1'' \left|#2\right|_p \else \left|#2\right|_{#1} \fi}

\renewcommand{\theta}{\vartheta}

\title{Logarithmic Sobolev inequality revisited}

\author[H.-M. Nguyen]{Hoai-Minh Nguyen}
\author[M.\ Squassina]{Marco Squassina}

\address[H.-M. Nguyen]{Department of Mathematics \newline\indent
	EPFL SB CAMA \newline\indent
	Station 8 CH-1015 Lausanne, Switzerland}
\email{hoai-minh.nguyen@epfl.ch}

\address[M.\ Squassina]{Dipartimento di Matematica e Fisica \newline\indent
	Universit\`a Cattolica del Sacro Cuore \newline\indent
	Via dei Musei 41, I-25121 Brescia, Italy}
\email{marco.squassina@unicatt.it}

\thanks{The second author is member of {\em Gruppo Nazionale per l'Analisi Ma\-te\-ma\-ti\-ca, la Probabilit\`a e le loro Applicazioni} (GNAMPA) of the {\em Istituto Nazionale di Alta Matematica} (INdAM)}



\begin{document}

\begin{abstract}
	We provide a new characterization of the logarithmic Sobolev inequality. 
\end{abstract}
\maketitle

\section{Introduction}

The classical Sobolev inequality translates information about the derivatives of a function
into information about the size of the function itself. Precisely, for a function $u$ with square
summable gradient in dimension $N$ one obtains that $u$ is $2N/(N-2)$-summable, that is a gain
in summability which depends on $N$ and which tends to deteriorate as $N\to\infty.$
On the other hand, since the middle fifties, people have started looking at possible replacements
of the Sobolev inequality in order to provide an improvement in the summability {\em independent} of  
the dimension $N$, which can be done in terms of integrability properties of $u^2 \log u^2$. 
This was firstly done by Stam \cite{stam} who proved the logarithmic Sobolev inequality with Gauss measure $d{\mathscr G}$
$$
\int_{\R^N} u^2\log \frac{u^2}{\|u\|_{2,d{\mathscr G}}^2} \,d{\mathscr G}\leq \frac{1}{\pi}\int_{\R^N}|\nabla u|^2 \,d{\mathscr G},
\qquad\,\,
d{\mathscr G}=e^{-\pi |x|^2} dx.
$$
The formula was originally discovered in quantum field theory in order to handle estimates 
which are uniform in the space dimension, for systems with a large number of variables.
A different proof and further insight was obtained by Gross in \cite{gross}.\ See also the work of Adams and Clarke \cite{adamsC} for an elementary proof of the previous inequality. 
These properties are widely used in statistical mechanics, quantum field theory and differential geometry.
A variant of the logarithmic Sobolev inequality with Gauss measure is given by the following 
one parameter family of euclidean inequalities \cite[Theorem 8.14]{LL}
$$
\int_{\R^N} u^2\log \frac{u^2}{\|u\|_{2}^2} \,dx+N(1+\log a)\|u\|_{2}^2\leq \frac{a^2}{\pi}\int_{\R^N}|\nabla u|^2 \,dx.
$$
for any $u\in H^1(\R^N)$ and $a>0$. A version of this inequality for fractional
Sobolev spaces $H^s(\R^N)$ can be found in \cite{costiolis}.
Recently some new characterization of the Sobolev spaces were provided in \cite{BourNg, nguyen06, nguyen07} (see also \cite{bourg,bourg2,bre,bre-linc,BHN,BHN2,BHN3, NgSob2})
in terms of the following family of nonlocal functionals
$$
I_\delta(u):=\int\int_{\{|u(y)-u(x)|>\delta\}}\frac{\delta^2}{|x-y|^{N+2}}dxdy,\quad \delta>0,
$$
where $u$ is a measurable function on $\R^N$. In particular, if $N\geq 3$ and  $I_\delta(u) < \infty$ for some $\delta>0$, then
in \cite{nguyen07} it was proved that 
\begin{equation}\label{inequality}
\int_{\{|u|>\lambda_N\delta\}} |u|^{2N/(N-2)}dx\leq C_N I_\delta(u)^{N/(N-2)}, 
\end{equation}
for some positive constants $C_N$ and $\lambda_N$. This is a sort of nonlocal improvement 
of the classical Sobolev inequality and it is also possible to show that in the singular limit $\delta\searrow 0$
one recovers the classical Sobolev result, since $I_\delta$ converges to the Dirichlet energy up to a normalization constant.
The aim of this note is to remark that in this context also a logarithmic type estimate holds. Thus we have the summability
gain independent of $N$ can be controlled in terms of $I_\delta(u).$

More precisely, we have the following 
\noindent
\begin{theorem}
	\label{main1}
	Let $u\in L^2(\R^N)$ $(N \ge 3)$.   There is a positive constant $C_N$ such that 
$$
\int_{\R^N} \frac{u^2}{\|u\|_{2}^2}\log \frac{u^2}{\|u\|_{2}^2} \,dx
+ \frac{N}{2}\log \|u\|_2^2\leq \frac{N}{2} \log \left( C_N \delta^\frac{4}{N} \|u\|_2^{\frac{2N - 4}{N}} + C_N I_\delta(u) \right), 
$$		
for all $\delta>0$. In particular, if $u\in L^2(\R^N)$ is such that $I_\delta(u) < \infty$ for some $\delta>0$, then
\begin{equation}
\label{log-sum}
\int_{\R^N} u^2\log u^2 dx<+\infty.
\end{equation}
\end{theorem}
\begin{proof}
	By a simple normalization argument, we may reduce the 
	assertion to proving that 
	\begin{equation}\label{part-1}
	\int_{\R^N} u^2\log u^2 dx\leq \frac{N}{2} \log \left(C_N  \delta^\frac{4}{N} + C_N I_\delta(u) \right),\quad \mbox{ for all } \delta > 0,
	\end{equation}
	for any $u\in L^2(\R^N)$ such that $\|u\|_{2}=1$.  Considering the normalized outer measure
	$$
	\mu(E):=\int_E u^2(x)dx,\qquad \mu(\R^N)=1,
	$$
	and using Jensen's inequality for concave nonlinearities and with measure $\mu,$
	we have 
	\begin{align}
	\label{chiav}
	 \log\left(\int_{\R^N}|u|^{\frac{2N}{N-2}}dx\right)&=	\log\left(\int_{\R^N}|u|^{\frac{4}{N-2}}d\mu\right) 
	\geq \int_{\R^N}\log |u|^{\frac{4}{N-2}} d\mu \\
	& =\frac{2}{N-2} \int_{\R^N}u^2\log u^{2}dx.   \notag
	\end{align}
	On the other hand, applying \eqref{inequality}, we derive that, for all $\delta > 0$, 
	\begin{equation*}
\frac{2}{N-2}	\int_{\R^N} u^2 \log u^2 dx\le \log \Big(D_N \delta^{\frac{4}{N-2}} +  C_N I_\delta (u)^{\frac{N}{N-2}}\Big),
	\end{equation*}
	for some positive constant $D_N$, which implies \eqref{part-1}.  Here we used the fact that 
	$$
	\int_{\{|u| \le \lambda_N \delta\}} |u|^{\frac{2N}{N-2}}dx \le \lambda_N^{\frac{4}{N-2}} \delta^{\frac{4}{N-2}},
	$$
	since $\int_{\R^N} u^2 \, dx = 1$. 
	\end{proof}

\noindent
Defining a notion of {\em entropy} as typical in statistical mechanics:
$$
{\rm Ent}_\mu(f):=\int_{\R^N} \frac{f}{\|f\|_{1,\mu}}\log \frac{f}{\|f\|_{1,\mu}} \,d\mu
+ \frac{N}{2}\log \|f\|_{1,\mu}, \,\,\,\quad f\geq 0,\quad \,\,\|f\|_{1,\mu}:=\int fd\mu,
$$
the conclusion of the previous results reads as
$$
u\in L^2(\R^N),\,\,\exists \delta>0: I_\delta(u)<+\infty\,\,\,\Longrightarrow\,\,\, {\rm Ent}_{{\mathcal L}^N}(u^2)<+\infty.
$$

\begin{remark}[Logarithmic NLS]
	\rm
If $u\in H^1(\R^N)$, then the results of \cite{nguyen06} show that
\begin{equation}
\label{approx}
\lim_{\delta\searrow 0} I_\delta(u)=Q_N\int_{\R^N}|\nabla u|^2dx,
\end{equation}
for some constant $Q_N>0$. Hence, passing to the limit as $\delta\searrow 0$ in the
inequality of Theorem~\ref{main1} one recovers classical forms of the logarithmic
inequality. The logarithmic Schr\"odinger equation
\begin{equation}
\label{sch-evolv}
\i\partial_{t}\phi +\Delta\phi 
+\phi\log |\phi|^{2}=0, 
\quad
\phi:[0,\infty)\times\R^N\to\C,
\quad
N\geq 3,
\end{equation}
admits applications to quantum mechanics, quantum optics, transport and diffusion 
phenomena, theory of superfluidity and Bose-Einstein condensation
(see \cite{Zlosh} and \cite{caz1,caz2,caz3}). The {\em standing waves} solutions of
\eqref{sch-evolv} solve
the following semi-linear elliptic problem
\begin{equation}\label{problema}
-\Delta u +\omega u
= u\log u^{2},  \qquad
u\in H^1(\R^N).
\end{equation}
These equations were recently investigated in \cite{troy,log1}.
From a variational point of view, the search of solutions 
to \eqref{problema} can be associated with the 
study of critical points (in a nonsmooth sense) of the lower 
semi-continuous functional $J:H^1(\R^N)\to \R\cup\{+\infty\}$ defined by
\begin{equation*}
J(u)=\frac{1}{2}\int_{\R^N} |\nabla u|^2\,dx 
+\frac{\omega+1}{2}\int_{\R^N} u^2\,dx-\frac{1}{2}\int_{\R^N} u^2 \log u^2\,dx,
\end{equation*} 
which is well defined by the logarithmic Sobolev inequality.
Due to Theorem~\ref{main1} and \eqref{approx}, one could handle 
a kind of {\em nonlocal approximations} of \eqref{problema}, 
formally defined for $\delta>0$ by 
\begin{equation*}
I'_\delta(u) +\omega u
= u\log u^{2},
\end{equation*}
which are associated with the energy functional
$J_\delta: H^1(\R^N)\to \R\cup\{+\infty\}$ defined by
\begin{equation*}
J_\delta(u)=I_\delta(u)+\frac{\omega+1}{2}\int_{\R^N} u^2\,dx-\frac{1}{2}\int_{\R^N} u^2 \log u^2\,dx.
\end{equation*} 
Since there holds $I_\delta(u)\leq C_{N}\int_{\R^N}|\nabla u|^2dx$ 
for all $\delta>0$ and $u\in H^1(\R^N)$ 
(cf.\ \cite[Theorem 2]{nguyen06}) the energy functional $J_\delta$ is well defined, for every $\delta>0$.
\end{remark}

\begin{remark}[Magnetic case]
	\rm 
If $A:\R^N\to\R^N$ is locally bounded and $u:\R^N\to\C$, we set
$$
\Psi_u(x,y):=e^{\i (x-y)\cdot A\left(\frac{x+y}{2}\right)}u(y),\quad x,y\in\R^N.
$$
It was observed in \cite{piemar} that the following {\em Diamagnetic inequality} holds 
	$$
	||u(x)|-|u(y)||\leq
	\big|\Psi_u(x,x)-\Psi_u(x,y)\big|,\quad\text{for a.e.\ $x,y\in\R^N$.}
	$$	
In turn, by defining
	$$
	I_\delta^A(u):=\int_{\{|\Psi_u(x,y)-\Psi_u(x,x)|>\delta\}}\frac{\delta^2}{|x-y|^{N+2}}dxdy,
	$$
we have
\begin{equation}
\label{dia}
	I_\delta(|u|)\leq 	I_\delta^A(u),\quad\text{for all $\delta>0$ and all measurable 
	$u:\R^N\to\C$.}
\end{equation}
Then, Theorem~\ref{main1} yields 
the following {\em Magnetic logarithmic Sobolev inequality}.
For $u\in L^2(\R^N)$, there is a positive constant $C_N$ such that 
$$
\int_{\R^N} \frac{|u|^2}{\|u\|_{2}^2}\log \frac{|u|^2}{\|u\|_{2}^2} \,dx
+ \frac{N}{2}\log \|u\|_2^2\leq \frac{N}{2} \log \left( C_N \delta^\frac{4}{N} \|u\|_2^{\frac{2N - 4}{N}} + C_N I_\delta^A(u) \right).
$$		
Notice that, since $I_\delta(|u|)\approx \|\nabla |u|\|_2^2$ as $\delta\searrow 0$ \cite{nguyen06} and 
$I_\delta^A(u)\approx \|\nabla u-\i Au\|_2^2$ as $\delta\searrow 0$ \cite{mag},
from inequality \eqref{dia} one recovers $\|\nabla |u|\|_2\leq\|\nabla u-\i Au\|_2$
which follows from the well-know diamagnetic inequality for the gradients
$|\nabla |u||\leq |\nabla u-\i Au|$, see \cite{LL}.
\end{remark}
\smallskip 

\noindent
As a companion to Theorem~\ref{main1}, we also have the following 
\noindent
\begin{theorem}
	\label{main2}
	Let $u\in L^2(\R^N)$ $(N \ge 3)$. Assume that there exists a non-decreasing 
	function $F:\R^+\to\R^+$ such that $F(ts)\leq t^\beta F(s)$ for any $s,t\geq 0$ and some $\beta>0$ and 
	\begin{equation}
	\label{integF}
	\int_{\R^{2N}}\frac{F(|u(x)-u(y)|)}{|x-y|^{N+2}}dxdy<+\infty.
	\end{equation}
	Then there exists a positive constant $C_{N,F}$ such that
	$$
	\int_{\R^N} \frac{u^2}{\|u\|_{2}^2}\log \frac{u^2}{\|u\|_{2}^2} \,dx
	+ \frac{N}{2}\log \|u\|_2^\beta\leq \frac{N}{2} \log \left(C_{N,F} \|u\|_2^\beta 
	+C_{N,F}\int_{\R^{2N}}\frac{F(|u(x)-u(y)|)}{|x-y|^{N+2}}dxdy\right), 
	$$		
	In particular, condition \eqref{log-sum} holds.
\end{theorem}
\begin{proof}
	Consider  the statement when $\|u\|_2=1$.
	In light of inequality \eqref{chiav}, since by \cite[Proposition 6]{nguyen07} there exists
	$C_N>0$ and $\lambda_N>0$ such that
\begin{equation}
\label{inequality2}
\int_{\{|u|>\lambda_N F(1/2)\}} |u|^{2N/(N-2)}dx\leq C_N \left(\frac{1}{F(1/2)} \int_{\R^{2N}}\frac{F(|u(x)-u(y)|)}{|x-y|^{N+2}}dxdy \right)^{N/(N-2)}, 
\end{equation}	
by arguing as in the previous proof, we get
\begin{equation*}
\frac{2}{N-2}	\int_{\R^N} u^2 \log u^2 \le \log \left(D_{N,F} + D_{N,F}\left( \int_{\R^{2N}}\frac{F(|u(x)-u(y)|)}{|x-y|^{N+2}}dxdy \right)^{N/(N-2)}\right),
\end{equation*}
where we used the fact that 
$$
\int_{\{|u| \le \lambda_N F(1/2)\}} |u|^{\frac{2N}{N-2}}dx \le \lambda_N^{\frac{4}{N-2}} F(1/2)^{\frac{4}{N-2}},
$$
since $\int_{\R^N} u^2 \, dx = 1$. Then, we get
\begin{equation*}
\int_{\R^N} u^2 \log u^2 \le \frac{N}{2}\log \Big(C_{N,F} + C_{N,F} \int_{\R^{2N}}\frac{F(|u(x)-u(y)|)}{|x-y|^{N+2}}dxdy\Big).
\end{equation*}
In the general case, using the sub-homogeneity condition on $F$ yields
\begin{equation*}
\int_{\R^N} \frac{u^2}{\|u\|_2^2} \log \frac{u^2}{\|u\|_2^2} \le \frac{N}{2}\log \Big(C_{N,F} + \frac{C_{N,F}}{\|u\|_2^\beta} \int_{\R^{2N}}\frac{F(|u(x)-u(y)|)}{|x-y|^{N+2}}dxdy\Big),
\end{equation*}
which yields the desired conclusion.
\end{proof}

\begin{remark}[$L^p(\R^N)$-version]
	\rm
If $p>1$ and $N>p$, one has a variant of \eqref{chiav}, namely
	\begin{equation}
\label{chiav2}
\log\left(\int_{\R^N}|u|^{\frac{Np}{N-p}}dx\right)
\geq \frac{p}{N-p} \int_{\R^N}|u|^p\log |u|^p dx.   
\end{equation}
Then, by arguing as in the proofs of Theorems~\ref{main1} and \ref{main2} with
\begin{equation}
\label{p-vers}
u\mapsto \int\int_{\{|u(y)-u(x)|>\delta\}}\frac{\delta^p}{|x-y|^{N+p}}dxdy, \quad\,\,
u\mapsto \int_{\R^{2N}}\frac{F(|u(x)-u(y)|)}{|x-y|^{N+p}}dxdy,
\end{equation}
in place of $I_\delta(u)$ and \eqref{integF} respectively, it is possible to get corresponding 
log-Sobolev inequalities as for the case $p=2$, via the results of \cite{nguyen07}. 
In particular, if $u\in L^p(\R^N)$ and 
the functionals  in \eqref{p-vers} are finite at $u$ for some $\delta>0$, then 
\begin{equation*}
\int_{\R^N} |u|^p\log |u|^p dx<+\infty.
\end{equation*}
The Euclidean logarithmic Sobolev inequalities for the $p$-case have been intensively studied,
see e.g. the work of Del Pino and Dolbeault \cite{delpino} and the references therein.
\end{remark}


\medskip
\begin{thebibliography}{99}
	
	
	\bibitem{adamsC}
	R.A.\ Adams, F.H.\ Clarke, 
	{\em Gross's logarithmic Sobolev inequality: a simple proof}, 
	Amer. J. Math. {\bf 101} (1979), 1265--1269. 
	
	\bibitem{bourg}
	J. Bourgain, H. Brezis, P. Mironescu, 
	{\it Another look at Sobolev spaces},
	in \emph{Optimal Control and Partial Differential Equations. A Volume in Honor of Professor Alain Bensoussan's 60th Birthday}
	(eds. J. L. Menaldi, E. Rofman and A. Sulem), IOS Press, Amsterdam, 2001, 439--455.
	
	\bibitem{bourg2}
	J. Bourgain, H. Brezis, P. Mironescu,
	{\it Limiting embedding theorems for $W^{s,p}$ when $s \uparrow 1$ and applications},
	{J. Anal. Math.} \textbf{87} (2002), 77--101.
	
	%
	
	\bibitem{BourNg}
J.~Bourgain and H-M. Nguyen, \emph{{A new characterization of Sobolev spaces}},
  C. R. Acad. Sci. Paris \textbf{343} (2006), 75-80.

	\bibitem{bre}
	H.\ Brezis,
	{\it How to recognize constant functions. Connections with Sobolev spaces},
	Russian Mathematical Surveys \textbf{57} (2002), 693--708.
	
	\bibitem{bre-linc}
	H.\ Brezis, 
	{\it New approximations of the total variation and filters in imaging}, Rend Accad. Lincei {\bf 26} (2015), 223--240.
	
	\bibitem{BHN}
	H.\ Brezis, H-M.  Nguyen,
	{\it Non-local functionals related to the total variation and connections with Image Processing}, preprint. 
	\url{http://arxiv.org/abs/1608.08204}
	
	\bibitem{BHN2}
	H.\ Brezis, H-M. Nguyen,
	{\it The BBM formula revisited}, Atti Accad. Naz. Lincei Rend.
	Lincei Mat. Appl. {\bf 27} (2016), 515--533. 
	
	\bibitem{BHN3}
	H. Brezis, H-M.  Nguyen,
	{\it Two subtle convex nonlocal approximations of the BV-norm},
	Nonlinear Anal. {\bf 137} (2016), 222--245. 
	
	%
	
	
	\bibitem{caz1} 
	T. Cazenave, {\it Stable solutions of the logarithmic Schr\"odinger equation}, Nonlinear Anal. {\bf 7} (1983), 1127--1140.
	
	\bibitem{caz2} 
	T. Cazenave, {\it An introduction to nonlinear
		Schr\"odinger equations}, Textos de M\'etodos Matem\'aticos
	{\bf 26}, Universidade Federal do Rio de Janeiro 1996.
	
	\bibitem{caz3} 
	T. Cazenave, A. Haraux, {\it \'Equations d'\'evolution avec non lin\'earit\'e logarithmique}, Ann. Fac. Sci. Toulouse Math.  {\bf 2} (1980), 21--51. 

\bibitem{costiolis}
A.\ Cotsiolis, N. Tavoularis, {\em On logarithmic Sobolev inequalities for higher 
	order fractional derivatives}, Comptes Rendus Acad.\ Sci. Paris 
{\bf 340} (2005), 205--208.

	\bibitem{log1}
	 P.\ d'Avenia, E.\ Montefusco, M. Squassina, 
	 {\em On the logarithmic Schr\"odinger equation}, Commun. Contemp. Math. {\bf 16} (2014), 1350032, 15 pp.
	 
	 \bibitem{piemar}
	 P.\ d'Avenia, M.\ Squassina,  
	 {\it Ground states for fractional magnetic operators},  
	 ESAIM COCV, to appear.
	
	\bibitem{delpino}
	M.\ Del Pino, J.\ Dolbeault, 
	{\em The optimal euclidean $L^p$-Sobolev logarithmic inequality}, 
	J. Funct. Anal.\ {\bf 197} (2003), 151--161.
	
	\bibitem{gross}
	L. Gross, {\em Logarithmic Sobolev Inequalities},
	Amer.\ J.\ Math.\ {\bf 97} (1975), 1061--1083.
	
	
	
	
	
	%
	
	\bibitem{LL}
	E.\ Lieb and M.\ Loss, Analysis, Graduate Studies in Mathematics {\bf 14}, AMS, 2001.
	
	%
	
	%
	
	\bibitem{nguyen06}
	H-M.\ Nguyen, {\it Some new characterizations of Sobolev spaces},
	J. Funct. Anal. \textbf{237} (2006), 689--720.
	
	\bibitem{NgSob2}
H-M.\ Nguyen, \emph{{Further characterizations of Sobolev spaces}}, J. Eur.\
  Math. Soc. \textbf{10} (2008), 191--229.
  
	\bibitem{nguyen07}
	H-M.\ Nguyen,
	{\it Some inequalities related to Sobolev norms}, 
	Calc. Var. PDEs {\bf 41} (2011), 483--509.
	

\bibitem{mag}
H-M.\ Nguyen, A. Pinamonti, M. Squassina, E. Vecchi,
{\em A new characterization of magnetic Sobolev spaces}, 
in preparation.

\bibitem{stam}
A.J. Stam, 
{\em Some inequalities satisfied by the quantities of information of Fisher and Shannon}, 
Information and Control {\bf 2} (1959), 101--112. 
	

\bibitem{troy}	
W.C.\ Troy, 
{\em Uniqueness of positive ground state solutions of the logarithmic Schr\"odinger equation},
Arch. Ration. Mech. Anal. {\bf 222} (2016), 1581--1600. 	
	
	\bibitem{Zlosh}
	K.G. Zloshchastiev, {\it Logarithmic nonlinearity in theories of quantum gravity: origin of time and observational consequences}, Grav. Cosmol. {\bf 16} (2010), 288--297.
	
	
\end{thebibliography}
\end{document}